\documentclass[11pt]{amsart}

\usepackage[all]{xy}
\usepackage{graphics}
\usepackage{color}
\usepackage[T1]{fontenc}
\usepackage{parskip}
\usepackage[colorlinks]{hyperref}
\usepackage{enumitem}
\usepackage{comment}

\usepackage{tikz}
\usepackage{tikz-cd}
\usetikzlibrary{decorations.markings}

\usepackage[colorinlistoftodos,prependcaption]{todonotes}

\newcommand{\B}[1]{{\mathbif #1}}
\newcommand{\C}[1]{{\mathcal #1}}

\newtheorem{theorem}{Theorem}[section]
\newtheorem{theorem*}{Theorem}
\newtheorem{lemma}[theorem]{Lemma}
\newtheorem{proposition}[theorem]{Proposition}
\newtheorem{corollary}[theorem]{Corollary}

\newtheorem*{question*}{Question}

\theoremstyle{definition}

\newtheorem*{remark*}{Remark}
\newtheorem*{remarks*}{Remarks}
\newtheorem*{corollary*}{Corollary}

\numberwithin{figure}{section}
\numberwithin{table}{section}
\numberwithin{equation}{section}

\def\B{\mathbf}

\newcommand{\OP}{\operatorname}

\begin{document}

\title{Entropy and quasimorphisms}
\author{Michael Brandenbursky}
\author{Micha\l\ Marcinkowski}
\address{Ben Gurion University of the Negev, Israel}
\email{brandens@math.bgu.ac.il}
\address{Regensburg Universit{\"a}t \& Uniwersytet Wroc\l awski}
\email{marcinkow@math.uni.wroc.pl}
\subjclass[2000]{Primary 53; Secondary 57}

\begin{abstract}
Let $S$ be a compact oriented surface. We construct homogeneous quasimorphisms 
on $\OP{Diff}(S, \OP{area})$, on $\OP{Diff}_0(S, \OP{area})$ and on $\OP{Ham}(S)$ 
generalizing the constructions of Gambaudo-Ghys and Polterovich.

We prove that there are infinitely many linearly independent 
homogeneous quasimorphisms on  $\OP{Diff}(S, \OP{area})$, on
$\OP{Diff}_0(S, \OP{area})$ and on $\OP{Ham}(S)$ 
whose absolute values bound from below the topological entropy. 
In case when $S$ has a positive genus, the quasimorphisms we construct
on $\OP{Ham}(S)$ are $C^0$-continuous.

We define a bi-invariant metric on these groups, called the entropy metric, 
and show that it is unbounded. In particular, we reprove the fact that the autonomous 
metric on $\OP{Ham}(S)$ is unbounded.
\end{abstract}

\maketitle

\tableofcontents

\section{Introduction}\label{S:intro} 
Let $M$ be a smooth compact manifold with some fixed
Riemannian metric. Let $f\colon M\to M$ be a continuous function. 
Recall that the topological entropy of $f$ may be defined as follows.
Let $\B d$ be the metric on $M$ induced by some Riemannian metric. For $p\in\B N$ define a new metric
$\B d_{f,p}$ on $M$ by 
$$\B d_{f,p}(x,y)=\max_{0\leq i\leq p} \B d(f^i(x),f^i(y)).$$
Let $M_f(p,\epsilon)$ be the minimal number of $\epsilon$-balls in the $\B d_{f,p}$-metric that cover $M$.
The topological entropy $h(f)$ is defined by
$$h(f)=\lim_{\epsilon\to 0}\limsup_{p\to\infty}\frac{\log{M_f(p,\epsilon)}}{p},$$
where the base of $\log$ is two.
It turns out that $h(f)$ does not depend on the choice 
of Riemannian metric, see \cite{Bowen, Dinaburg}.

It  is notoriously difficult to compute topological entropy of a given diffeomorphism, 
even to detect whether entropy of a given diffeomorphism 
is non zero is a difficult task in most cases.

We consider the case when $M$ is a compact connected 
oriented surface $S$ endowed with an area form. Let
$\OP{Diff}(S,\OP{area})$ and $\OP{Diff}_0(S,\OP{area})$ be groups, 
where the first group is the group of area-preserving diffeomorphisms of $S$, 
and the second group is the group of area-preserving 
diffeomorphisms of $S$ isotopic to the identity.
If $S$ has a boundary, then we assume that diffeomorphisms 
are identity in some fixed neighborhood of the boundary. 

In the first part of the paper we revise and extend the 
construction of quasimorphisms on 
$\OP{Diff_0}(S,\OP{area})$ given by Gambaudo-Ghys \cite{MR2104597} and Polterovich \cite{MR2276956}, see
also \cite{MR2851716, MR3391653}. 
The main advantage of our approach is that it allows to treat all surfaces  
in an unified way and to show there are infinitely many linearly independent homogeneous
quasimorphisms on $\OP{Diff}(S, \OP{area})$ whose restrictions on $\OP{Diff}_0(S, \OP{area})$
are linearly independent. 

In the second part of our work we show that there are infinitely many
linearly independent homogeneous quasimorphisms on 
$\OP{Diff}(S,\OP{area})$ and on $\OP{Diff}_0(S,\OP{area})$
whose absolute values bound from below the topological entropy.
The same holds for the group $\OP{Ham}(S)$ of Hamiltonian diffeomorphisms of $S$.
More precisely, we apply the construction described in Section \ref{S:qm} to
quasimorphisms on mapping class groups constructed by Bestvina and Fujiwara in \cite{MR1914565}. 
We prove that these quasimorphisms are Lipschitz with respect to the topological entropy.
Our work is inspired by the paper of Gambaudo and Pecou \cite{MR1695912}
who constructed a dynamical cocycle on the group $\OP{Diff}(D^2, \OP{area})$ 
which bounds from below the topological entropy.

Recall that a function $\psi$ from a group $G$ to the reals
is called a quasimorphism if there exists $D$
such that 
$$|\psi(a)-\psi(ab)+\psi(b)|<D$$ 
for all $a,b\in G$. 
Minimal such $D$ is called the defect of $\psi$ and denoted by $D_\psi$.
A quasimorphism $\psi$ is homogeneous if $\psi(a^n)=n\psi(a)$ 
for all $n\in\B Z$ and $a\in G$. 
Quasimorphism $\psi$ can be homogenized by setting
$$\overline{\psi} (a) := \lim_{p\to\infty} \frac{\psi (a^p)}{p}.$$
The vector space of homogeneous quasimorphisms on $G$ is denoted by $Q(G)$. 
For more information about quasimorphisms and their 
connections to different branches of mathematics, see \cite{Calegari}.
Throughout the paper we assume that the surface $S$ is always connected. 
Our main result is the following

\begin{theorem*}\label{T:main}
Let $S$ be a compact oriented Riemannian surface and let the group
$G=\OP{Diff}(S,\OP{area})$ or $G=\OP{Diff}_0(S,\OP{area})$ or $G=\OP{Ham}(S)$.
Then there exists an infinite dimensional subspace of $Q(G)$
such that every $\Psi$ in this subspace is Lipschitz with respect to the topological entropy, i.e., 
there exists a positive constant $C_{\Psi}$, which depends only on $\Psi$,  
such that for every $f\in G$ we have $$|\Psi(f)|\leq C_{\Psi}h(f).$$
\end{theorem*}

Let $\OP{Ent}(S)\subset G$ be the set of entropy-zero diffeomorphisms.
This set is conjugation invariant and it generates $G$, see Lemma \ref{L:generation}.  
In other words, a diffeomorphism of $S$ is a finite product of entropy-zero diffeomorphisms. 
One may ask for a minimal decomposition and this question
leads to the concept of the entropy norm which we define by
$$
\|f\|_{\OP{Ent}}:=\min\{k\in\B N\,|\,f=h_1\cdots h_k,\,h_i\in\OP{Ent}(S)\}.
$$
It is the word norm associated with the generating set $\OP{Ent}(S)$.
This set is conjugation invariant, so is the entropy
norm. The associated bi-invariant metric is denoted by $\B d_{\OP{Ent}}$.
It follows from the work of Burago-Ivanov-Polterovich \cite{MR2509711}
and Tsuboi \cite{MR2523489,MR2874899} that for many manifolds all conjugation
invariant norms on $\OP{Diff}_0(M)$ are bounded. Hence the
entropy norm is bounded in those cases.

We show that the situation is different for $G$.
More precisely, as a corollary of our main result we obtain the following

\begin{theorem*}\label{T:unboundedness}
Let $S$ be a compact oriented Riemannian surface and let the group
$G=\OP{Diff_0}(S,\OP{area})$ or $G=\OP{Diff}(S,\OP{area})$ or $G=\OP{Ham}(S)$. 
Then the diameter of $(G, \B d_{\OP{Ent}})$ is infinite.
Moreover, in case when $S$ is a closed disc, 
for each $m\in\B N$ there exists a bi-Lipschitz injective homomorphism 
$$\B Z^m\hookrightarrow (G, \B d_{\OP{Ent}}),$$
where $\B Z^m$ is endowed with the $l^1$-metric.
\end{theorem*}

\begin{remark*} There exists another conjugation invariant word norm on 
$\OP{Ham}(S)$, the autonomous norm.
It is unbounded in the case when $S$ is a compact oriented surface, 
see \cite{MR3391653, MR3044593, BKS, 1405.7931, MR2104597}.
Theorem \ref{T:unboundedness} together with the fact that 
every autonomous diffeomorphism of a surface has entropy zero, see \cite{MR0464325}, gives a new proof
of unboundedness of the autonomous norm on $\OP{Ham}(S)$.
\end{remark*}

\subsection* {Acknowledgments.}
We would like to thank Mladen Bestvina, Danny Calegari and Benson Farb for fruitful conversations. 
We would like also to thank Bernd Ammann and Ulrich Bunke for spotting an inaccuracy in the earlier proof of Lemma \ref{l:fin.many.braids}. 

Both authors were partially supported by GIF-Young Grant number I-2419-304.6/2016
and by SFB 1085 ``Higher Invariants'' funded by DFG. 
Part of this work has been done during the first author stay at the 
University of Regensburg and the second author stay
at the Ben Gurion University. We wish to express our gratitude 
to both places for the support and excellent working conditions.

\section{Quasimorphisms on diffeomorphisms groups of surfaces} \label{S:qm}

Let $S$ be a compact oriented surface endowed with an area form. In this section we define two linear maps
$$\C G_{S,n} \colon Q(\OP{MCG}(S,n)) \to Q(\OP{Diff}(S, \OP{area})),$$
$$\C G^0_{S,n} \colon Q(\OP{K}(S,n)) \to Q(\OP{Diff}_0(S, \OP{area})).$$
Here $\OP{MCG}(S,n)$ is the mapping class group of $S$ with $n$ punctures 
and $\OP{K}(S,n)$ is a certain subgroup of $\OP{MCG}(S,n)$.
If $S$ is a disk, then the group $\OP{K}(S,n)$ is isomorphic to the braid group on $n$ strands 
and $\C G^0_{S,n}$ is the map defined by Gambaudo-Ghys \cite{MR2851716, MR2104597}.
If $S$ is a closed surface of genus $g\geq 2$, 
then $\OP{K}(S,1)$ is isomorphic to the fundamental group of $S$ 
and $\C G^0_{S,1}$ is the map defined by Polterovich \cite{MR2276956}. 

The main difference in our approach is that we work with mapping class groups, instead of braid groups. 
Our definition of $\C G^0_{S,n}$ is isotopy free, which allows us to define 
$\C G^0_{S,n}$ in the case when $\pi_1(\OP{Diff}_0(S,\OP{area}))$ is non-trivial. 
It also allows us to extend the construction to the case of $\OP{Diff}(S,\OP{area})$.

The map $\C G_{S,n}$ is an extension of $\C G^0_{S,n}$ in a sense that we have the
following commutative diagram
$$
\begin{tikzcd}
Q(\OP{MCG}(S,n))\arrow{d}\arrow{r}{\C G_{S,n}}& Q(\OP{Diff}(S,\OP{area})) \arrow{d}\\
Q(\OP{K}(S,n))\arrow{r}{\C G^0_{S,n}}& Q(\OP{Diff}_0(S, \OP{area}))
\end{tikzcd}
$$
The vertical maps are restrictions to subgroups.
Many quasimorphisms we construct on $\OP{Diff}(S,\OP{area})$
restrict to non-trivial quasimorphisms on $\OP{Diff}_0(S,\OP{area})$.
In particular, they do not arise as a pull back of quasimorphims on $\OP{MCG}(S)$ by the
quotient map 
$$\OP{Diff}(S,\OP{area})\to\OP{MCG}(S).$$

In the remaining part of this section we define 
maps $\C G_{S,n}$, $\C G^0_{S,n}$ and discuss their basic properties.

\subsection{Configuration space}\label{sec:conf}

Let $D^2$ be an open disc in the Euclidean plane. 
Let $X_n$ be the configuration space of $n$ points in $D^2$. 
We fix a tuple $z~=~(z_1,\ldots,z_n)\in X_n$. 

Let $ev_{z}\colon \OP{Diff}(D^2)\to X_n$ be defined by $f\to f(z)=(f(z_1),\ldots,f(z_n)).$

It is shown in \cite[Section 3.2 and Theorem 4]{MR1695912}, 
that there is a subset $H_n$ of $X_n$ with the following properties: 
it is a union of submanifolds of codimension $1$, 
there exists a map $h\colon X_n\backslash H_n\to\OP{Diff}(D^2)$ 
which is a section of $ev_{z}$, i.e., $ev_{z}\circ h$ is the identity on $X_n\backslash H_n$.

Denote by $\Omega_n=X_n\backslash H_n$, and let $h$ be a section described as follows: 
let $x=(x_1,\ldots,x_n)\in\Omega_n$. 
Let $P_i$ be a geodesic segment connecting $z_i$ to $x_i$. 
The fact that $x\notin H_n$ guaranties, 
that $x_i$ and $z_i$ do not lie on $P_j$ for 
every $i$ and $j \neq i$ (for details see \cite[Section 3.2]{MR1695912}).
By $N_\epsilon(P_i)$ we denote the $\epsilon$-neighborhood of $P_i$. 
By the definition of $H_n$, we can pick a small $\epsilon(x)\in\B R_+$ 
such that $x_i,z_i\notin N_{\epsilon(x)}(P_j)$ for every $i$ and $j\neq i$. 
The choice of $\epsilon(x)$ can be made such that 
the function $\epsilon(x)$ is $C^1$-continuous on $\Omega_n$. 

Let $h_{x_i}\in\OP{Diff}(D^2)$ be a map which pushes $z_i$ to $x_i$ along $P_i$ 
and is supported on $N_{\epsilon(x)}(P_i)$.
We set $h_x=h_{x_1}\circ h_{x_2}\circ\ldots\circ h_{x_n}$. 
By the definition, $h_x$ maps $z_i$ to $x_i$. 
This is a $C^1$-continuous section of $ev_{z}$. 

\subsection{The cocycle}

Let $S$ be a compact oriented surface with an area form.
We take a map $j\colon D^2\to S$ which is an attachment of an open $2$-cell
to the $1$-skeleton of $S$. The image of $D^2$ is of full measure in $S$. 
In what follows we always regard $D^2$ as a subset of $S$.

The area form on $S$ induces the volume form on $X_n(S)$, 
which is the configuration space of $n$-points in $S$.
Spaces $X_n=X_n(D^2)$ and $\Omega_n$ are full measure subsets of $X_n(S)$.
The group $\OP{Diff}(S,\OP{area})$ acts on $X_n(S)$ preserving the measure. 
Let $\OP{MCG}(S,n)=\OP{MCG}(S,\{j(z_1),\ldots,j(z_n)\})$, 
where $\{z_i\}$ are defined in Section \ref{sec:conf}.
We define a cocycle
$$\gamma_{S,n}\colon\OP{Diff}(S,\OP{area})\times X_n\to\OP{MCG}(S,n)$$
by the formula 
$$\gamma_{S,n}(f,x)=[h^{-1}_{f(x)}\circ f\circ h_x].$$
To be fully correct the map $\gamma_{S,n}(f)\colon X_n\to X_n$ is not defined 
on $X_n$, but on a full measure subset of $X_n$ which depends on $f$,
namely on the set $\Omega_{n,f}~=~\Omega_n\cap f^{-1}(\Omega_n)$.
It is easy to show that $\gamma_{S,n}$ is a cocycle, i.e., 
$$\gamma_{S,n}(fg,x)=\gamma_{S,n}(f,g(x))\gamma_{S,n}(g,x).$$ 
 
Consider the forgetful map 
$$F\colon\OP{MCG}(S,n)\to\OP{MCG}(S)$$ 
and denote $\OP{K}(S,n)=Ker(F)$. If $f\in\OP{Diff}_0(S,\OP{area})$ then 
$$\gamma_{S,n}(f,x)=[h^{-1}_{f(x)}\circ f\circ h_{x}]$$ 
is homotopic to the identity in $S$, possibly by a homotopy which can move points $j(z_i)$. 
Thus $F(\gamma_{S,n}(f,x))=1$ and $\gamma_{S,n}(f,x)\in\OP{K}(S,n)$. 
It follows that we can restrict $\gamma_{S,n}$ 
to $\OP{Diff}_0(S,\OP{area})$ and obtain the cocycle: 
$$\gamma^0_{S,n}\colon\OP{Diff}_0(S,\OP{area})\times X_n\to\OP{K}(S,n).$$
In the same way $\gamma_{S,n}$ restricts to $\OP{Ham}(S)$.

\subsubsection{Relation to braids}\label{s:relation.to.braids}

Let us recall a construction due to Gambaudo and Ghys \cite[Section 5.2]{MR2104597}. 
With an isotopy $f_t$, $t \in [0,1]$ and a point $x = (x_1,\ldots,x_n) \in X_n(S)$, 
we associate a braid $\gamma'_{S,n}(f,x) \in \B B_n(S)$ in the following way: 
we connect $z$ with $x$ using geodesic segments as in Section \ref{sec:conf},
then we connect $x$ with $f(x)$ by $f_t(x)$, $t \in [0,1]$ and at the end we connect $f(x)$ again with $z$ using geodesic segments.

Let us now describe the relation between $\gamma'_{S,n}(f,x)$ and $\gamma_{S,n}(f,x)$. Recall the Birman map:
$$
Push \colon \B B_n(S) \to \OP{MCG}(S,n),
$$
where $\B B_n(S) = \pi_1(X_n(S),z)$ is the braid group of $S$ on $n$ strings and the definition of $Push$ is the following: 
let $\gamma(t)$, $t\in[0,1]$, be a loop in $X_n(S)$ based at $z$ and $\psi_t \in \OP{Diff}(S)$ an isotopy such that $\psi_t(z) = \gamma(t)$. 
Then $Push([\gamma]) = [\psi_1]$. From this description of $Push$ it is immediate that $Push(\gamma'_{S,n}(f,x)) = \gamma_{S,n}(f,x)$. 

\subsubsection{Finitely many mapping classes.}
We say that a function $\gamma$ defined on a probability space $X$ has essentially finite image, if there exists
a full measure subset of $X$ on which $\gamma$ has finite image.

\begin{lemma}\label{l:fin.many.braids}
For given $f\in\OP{Diff}(S,\OP{area})$, the map 
$$\gamma_{S,n}(f)\colon X_n\to\OP{MCG}(S,n)$$ 
has essentially finite image. 
\end{lemma}

Before the proof we need some preparations.
The following Proposition is an immediate consequence of the cocycle condition.

\begin{proposition}\label{prop:products} 
Let $f_i \in\OP{Diff}(S,\OP{area})$, $i = 1,\ldots,n$. Assume that functions $\gamma_{S,n}(f_i)$ have essentially finite images. 
Then $f_1f_2\ldots f_n$ has essentially finite image. 
\end{proposition}

To prove Lemma \ref{l:fin.many.braids} it is enough to prove it for some generating set of $\OP{Diff}(S,\OP{area})$. 
Let us consider the following three types of diffeomorphisms.

\begin{itemize}
\item Morse autonomous diffeomorphisms: 
let $H$ be a Morse function on $S$. 
A Morse autonomous diffeomorphism $f$ is the Hamiltonian  
diffeomorphism defined by $H$, i.e. $f$ is the time-one map of the flow $f_t$ given by a vector field $X_H$, 
where $X_H$ is defined by the equation $dH = \iota_{X_H}\OP{area}$. 
\item Hamiltonian pushes: 
let $\sigma$ be a simple loop in $S$ and let
$A$ be a tubular neighborhood of $\sigma$. 
A Hamiltonian push is an element in $\OP{Diff_0}(S,\OP{area})$ which is the identity on the complement of $A$ 
and when restricted to $A\cong [0,1] \times S^1$ it is a time-$t$ map for some $t \in \B R$ 
of a Hamiltonian defined by $H(s,\psi) = g(s)$ where 
$g$ is a monotone function such that $g(\delta)=0$ and $g(1-\delta)=1$ for all $\delta < \frac{1}{3}$. 
\item Area-preserving Dehn twists: Let $S^1 = \B R/\B Z$. 
The standard Dehn twist of the annulus $[0,1]\times S^1$ is given by $D(s,\psi)=(s,\psi+f(s))$,
where $f \colon [0,1] \to [0,1]$ is any smooth monotone function such that $f(0) = 0$ and $f(1) = 1$.
Note that $D$ preserves the Lebesgue measure on $[0,1]\times S^1$.
Let $\sigma$ be a simple loop in $S$ and let $A$ be a tubular neighbourhood of $\sigma$. 
An area-preserving Dehn twist associated to $\sigma$ is a map which is the identity on the complement of $A$ and on $A$ it is
the pull-back of $D$ by some area-preserving diffeomorphism between $A$ and $[0,1]\times S^1$.
\end{itemize}


\begin{lemma}\label{l:generators}
Let $S$ be a closed oriented surface. Then Morse autonomous diffeomorphisms, Hamiltonian pushes and area-preserving Dehn twists generate $\OP{Diff}(S,\OP{area})$.
\end{lemma}

\begin{proof}
Note that the set of Morse autonomous diffeomorphisms is a conjugacy invariant subset of $\OP{Ham}(S)$. 
If follows from the simplicity of $\OP{Ham}(S)$ that this set generates $\OP{Ham}(S)$. 
Now consider the flux homomorphism
$$Flux \colon \OP{Diff_0}(S,\OP{area}) \to H^1(S,\B R)/\Gamma.$$
It is known that $Ker(Flux) = \OP{Ham}(S)$ and for every $c \in H^1(S,\B R)/\Gamma$ one can find
a product of Hamiltonian pushes $p$ such that $Flux(p) = c$. Thus Morse autonomous diffeomorphisms 
and Hamiltonian pushes generate $\OP{Diff}_0(S,\OP{area})$.
Recall that 
$$\OP{MCG}(S) = \OP{Diff}(S,\OP{area})/\OP{Diff}_0(S,\OP{area}).$$
Now the Lemma follows from the fact that $\OP{MCG}(S)$ is generated by mapping classes of area-preserving Dehn twists. 
\end{proof}

\begin{proof}[Proof of Lemma \ref{l:fin.many.braids}]
First we consider the case when $S$ is a closed oriented surface. 
It follows from Lemma \ref{prop:products} that it is enough to prove the statement for 
Morse autonomous diffeomorphisms, Hamiltonian pushes and area-preserving Dehn twists.

Let $f$ be a Morse autonomous diffeomorphism. There exists a full measure subset $X_n^0(S) < X_n(S)$ where
the set of braids associated to $f$ is finite, see \cite[Section 2.C]{MR3391653}. It means that
the set $\{ \gamma'_{S,n}(f,x)~|~x\in X_n^0(S)\}$ is finite.

The same analysis as in \cite[Section 2.C]{MR3391653} is applied to Hamiltonian pushes.  

In the case of area-preserving Dehn twists we proceed as follows.
Let $f$ be a Dehn twist supported in annulus $A \subset S$. 
We can assume that $z_i \notin A$ for all $i=1,\ldots,n$. 
Thus $[f] \in \OP{MCG}(S,n)$ and  
$$\gamma_{S,n}(f,x)[f^{-1}] = [h^{-1}_{f(x)}\circ f \circ h_x \circ f^{-1}] \in \OP{MCG}(S,n).$$
Since $h_x$ and $h_{f(x)}$ are isotopic to the identity, this implies that 
$$\gamma_{S,n}(f,x)[f^{-1}] \in im(Push) = K(S,n).$$ 
Let $x=(x_1,\ldots,x_n) \in X_n(S)$ and $P_{z_i,x_i}$ an interval connecting $z_i$ with $x_i$ as in Section \ref{sec:conf}. 
Note that $f \circ h_x \circ f^{-1}$ is a diffeomorphism which pushes $z_i$ to $f(x_i)$ along the curve $f(P_{z_i,x_i})$.
Let $\delta_{S,n}(f,x)$ be a braid described as follows: first we connect $z_i$ with $f(x_i)$ by curves $f(P_{z_i,x_i})$ 
and then we connect $f(x_i)$ with $z_i$ by $P_{z_i,f(x_i)}$. 
It follows from the definition of $Push$ in Subsection \ref{s:relation.to.braids}, that $Push(\delta_{S,n}(f,x)) = \gamma_{S,n}(f,x)[f^{-1}]$. 
Now the same analysis as in \cite[Section 2.C]{MR3391653} shows that there are finitely many braids of the form $\delta_{S,n}(f,x)$, 
and the proof follows for closed $S$.

Assume that $S$ has a boundary. 
In this case we embed $S$ into a closed surface $\overline{S}$ such that $i \colon \OP{MCG}(S,n) \to \OP{MCG}(\overline{S},n)$
is an embedding. For example, one can cap each boundary component of $S$ with a torus with one boundary component.  
We extend the area form from $S$ to $\overline{S}$.
It is possible to define the geodesic segments $P_i$ for $\overline{S}$ such that they 
agree with the geodesic segments defined for $S$ (see Section \ref{sec:conf}).
Now for $x \in S$ and $f \in \OP{Diff}(S,\OP{area})$ we have that 
$i \circ \gamma_{S,n}(f,x) = \gamma_{\overline{S},n}(\bar f,x)$, where $\bar f$ is the extension of $f$
to $\OP{Diff}(\overline{S},\OP{area})$ by the identity.  
\end{proof}
  
\subsection{Definition of $\C G_{S,n}$ and $\C G^0_{S,n}$}
Let $\psi\in Q(\OP{MCG}(S,n))$ and a diffeomorphism $f$ in $\OP{Diff}(S,\OP{area})$. 
By Lemma \ref{l:fin.many.braids} the function $x\to\psi\circ\gamma_{S,n}(f,x)$ is integrable. 
We define
$$\C G'_{S,n}(\psi)(f)=\int_{X_n}\psi\circ\gamma_{S,n}(f,x) dx\thinspace.$$

\begin{lemma}
The function $\C G'_{S,n}(\psi)$ is a quasimorphism.
\end{lemma}
\begin{proof}
\begin{align*}
\C G'_{S,n}(\psi)(fg) &= \int_{X_n}\psi\circ\gamma_{S,n}(fg,x) dx\\
&= \int_{X_n}\psi(\gamma_{S,n}(f,g(x))\gamma_{S,n}(g,x)) dx\\
&\leq \int_{X_n}\psi\circ\gamma_{S,n}(f,g(x))+\psi\circ\gamma_{S,n}(g,x)+D_{\psi} dx\\
&=\C G'_{S,n}(\psi)(f)+\C G'_{S,n}(\psi)(g)+\OP{Area}(S)D_{\psi}\thinspace.
\end{align*}
In the last equality we used the fact that $g$ preserves the measure, 
thus $\C G'_{S,n}(\psi)(f)=\int_{X_n}\psi\circ\gamma_{S,n}(f,g(x))dx$.
In a similar way one shows that 
$\C G'_{S,n}(\psi)(fg)\geq \C G'_{S,n}(f)+\C G'_{S,n}(g)-\OP{Area}(S)D_{\psi}$.
\end{proof}

We define $\C G_{S,n}(\psi)$ to be the stabilization of $\C G'_{S,n}(\psi)$, i.e.,
$$\C G_{S,n}(\psi)(f)=\lim_{p\to\infty}\frac{\C G'_{S,n}(\psi)(f^p)}{p}\thinspace,$$
$$\C G_{S,n}\colon Q(\OP{MCG}(S,n))\to Q(\OP{Diff}(S,\OP{area})).$$

The map $\C G^0_{S,n}$ is defined in the same way as $\C G_{S,n}$, 
except that now instead of $\gamma_{S,n}$ we use $\gamma^0_{S,n}$.
In this situation we obtain a linear map
$$\C G^0_{S,n} \colon Q(\OP{K}(S,n))\to Q(\OP{Diff}_0(S,\OP{area})).$$
It is also defined from $Q(\OP{K}(S,n))$ to $Q(\OP{Ham}(S))$.

\subsection{Embedding Theorem}
The proof of the following theorem is a variation  
of the proof of Ishida \cite{MR3181631}, see also \cite{MR3391653, BKS}.
We present it for the reader convenience.

\begin{theorem}\label{T:embed}
$\C G_{S,n}$ and $\C G^0_{S,n}$ are injective.
\end{theorem}
\begin{proof}
We give a proof in the case of $\C G_{S,n}$. 
The argument for the injectivity of $G^0_{S,n}$ goes along the same lines. 
Let $\psi\in Q(\OP{MCG}(S,n))$ and let $\gamma$ in $\OP{MCG}(S,n)$ such that $\psi(\gamma)\neq 0$.
Dehn twists generate $\OP{MCG}(S,n)$, thus we can express $\gamma$ as a product 
of Dehn twists along some finite set of simple loops $\C C$. 
We assume that $z_i$ does not lie on any loop in $\C C$. 
Let $N$ be a small tubular neighborhood of loops in $\C C$ such that $z_i\notin N$. 
We choose $f$ such that $[f]=\gamma$ and $f$ is supported in $N$. 
The idea of the proof is to show that we can choose $N$ in a way that $\C G_{S,n}(\psi)(f) \neq 0$. 
In what follows we will split $X_n$ into two pieces: one which has a small volume, 
and one on which $f$ is the identity. This allows us to control the value of the integral.  

\textbf{Step 1.} 
By definition $f$ is the identity on $D^2 \backslash N$. 
Let $X_N$ be a set of tuples in $X_n$ which have at least one coordinate in $N$.
Then the volume of $X_N$ goes to zero when the area of $N$ goes to zero. 
Let $\gamma=\gamma_{S,n}$ and $D_f$ be a full measure subset of $X_n$ on which $\gamma(f)$ has finite image.
Let
$$C=\sup\{|\psi(\gamma(f,x))|\phantom{.}|\phantom{.} x\in D_f\}.$$ 
For all $x$ in certain full measure subset of $X_N$ we have:
\begin{align*}
\frac{\psi(\gamma(f^p,x))}{p}=\frac{\psi(\gamma(f,f^{p-1}(x))\ldots \gamma(f,x))}{p}\leq\frac{pC + pD_{\psi}}{p}=C + D_\psi. 
\end{align*}
Thus
$$A_N=\int_{X_N} \lim_{p\to\infty}\frac{\psi(\gamma(f^p,x))}{p}dx\leq \OP{vol}(X_N)(C+D_\psi).$$

\textbf{Step 2.} Let $U=D^2\backslash N$ and consider $U^n\subset X_n$. 
Note that $X_n=U^n\cup X_N$ and $f$ acts identically on $U^n$.
The set $U$ is open and has finitely many connected components. Denote them by $U_1,\ldots,U_k$. 
Let $x \in U^n$. By $c_j(x)$ we denote the number of coordinates of $x$ which belong to $U_j$.
The following two claims show that  $\psi(\gamma(f,x))$ depends only on the numbers $\{c_j(x)\}_{j=1}^k$.

\textbf{Claim 1.} Let $x,y \in U^n$. Assume that $x_i = y_{\sigma(i)}$ 
for some permutation $\sigma \in \OP{Sym}_n$ and $i = 1,\ldots,n$.
Then $\gamma(f,x)$ and $\gamma(f,y)$ are conjugated in $\OP{MCG}(S,n)$.

\begin{proof}
Consider a map $h^{-1}_xh_y$. This map permutes the points $z_i$, thus $[h^{-1}_xh_y] \in \OP{MCG}(S,n)$.
Then $[h^{-1}_xh_y]\gamma(f,y)[h_y^{-1}h_x] = \gamma(f,x)$.
\end{proof}

\textbf{Claim 2.} Let $x,y\in U^n$ such that $x_i, y_i$ belong to the 
same connected component of $U$ for $i=1,\ldots,n$. 
Then $\gamma(f,x)$ and $\gamma(f,y)$ are conjugated in $\OP{MCG}(S,n)$.

\begin{proof}
Let $g$ be a diffeomorphism of $S$ such that $g(x_i) = y_i$ and $g$ is supported on $U$. 
In particular $g$ can be taken to be a map which pushes $x_i$ towards $y_i$ 
and $x_i$ travels all the time in the same connected component of $U$.
We consider the mapping class $[h^{-1}_y\circ g\circ h_x]\in\OP{MCG}(S,n)$. 
The maps $g$ and $f$ have disjoint supports, hence they commute.
Then 
$$[h^{-1}_x\circ g^{-1}\circ h_y]\gamma(f,y)[h^{-1}_y\circ g\circ h_x]
=[h^{-1}_x\circ g^{-1}\circ f\circ g\circ h_x]=\gamma(f,x).$$
\end{proof}

The set $U^n$ splits into finitely many connected components 
of the form $U_s=U_{s(1)}\times\ldots\times U_{s(n)}$, where
$s\colon \{1,\cdots,n\} \to \{1,\cdots,k\}$.
Let $C=(c_1,\ldots,c_k)$ be a partition such that $c_1+\ldots+c_k=n$. 
Consider a component $U_s=U_{s(1)}\times \ldots \times U_{s(n)}$ of $U^n$
for which $c_j=\#s^{-1}(j)$. We say that $U_s$ is associated to $C$. 
The function $\psi(\gamma(f))$ is constant on connected components associated to $C$,
and on each component it has the same value. Denote it by $\psi(\gamma(f,C))$. 
Let $L(C)$ be the number of connected components associated to $C$. 
Every connected component associated to $C$ has the same volume 
$\OP{vol}(C)=\OP{area}(U_1)^{c_1}\ldots\OP{area}(U_k)^{c_k}$.

Recall that $\gamma(f,z)=[f]$ and $\psi([f])\neq 0$. 
Let $C_0=(c_1,\ldots,c_k)$ be the partition corresponding to $z$, that is $c_i$ is the number 
of coordinates of $z$ which lie in $U_i$. Then $\psi(\gamma(f,C_0))=\psi([f])\neq 0$. 
Since $f$ is the identity on $U$, we have $\gamma(f^p,x)=\gamma(f,x)^p$.
Now we compute:
\begin{align*}
B_N&=\int_{U^n}\lim_{p\to\infty}\frac{\psi(\gamma(f^p,x))}{p}dx=\int_{U^n}\psi(\gamma(f,x))dx\\
&=\sum_{C\colon c_1+\ldots+c_k=n}\psi(\gamma(f,C)) L(C)\OP{vol}(C).
 \end{align*}

If we treat $\OP{area}(U_i)$ as a free variable, 
this integral is a homogeneous polynomial in $k$ variables of degree $n$.
Denote this polynomial by $P$. The coefficient of the monomial $\OP{vol}(C_0)$ 
is $\psi(\gamma(f,C_0))L(C_0)$ and is non-zero. 
The $\OP{area}(U_i)$ depends on the neighborhood $N$. 
If we start shrinking $N$ such that $\OP{area}(N)$ converges to zero, 
then $\OP{area}(U_i)$ converges to some number $V_i$. 
We can assume that $P(V_1,\ldots,V_k)\neq 0$.
Indeed, if it is not the case, we can modify a little the loops $\C C$ which were 
chosen to express $[f]$ in terms of Dehn twists. 
Then the values of $V_i$ change freely, except that we 
have a constraint $V_1+\ldots+V_k=\OP{area}(S)$.
It is easy to see, that a homogeneous polynomial $P$ is 
non-trivial on every affine non-linear subspace of codimension one, 
so we can arrange $V_i$ such that $P(V_1,\ldots,V_k)\neq 0$.

Now we shrink $N$, then $\OP{area}(N)\to 0$ and $\OP{area}(U_i)\to V_i$. 
We have that $A_N\to 0$ and $B_N\to P(V_1,\ldots,V_k)\neq 0$. 
Since $\C G_{S,n}(\psi)(f)=A_N+B_N$, then for some $N$ we have $\C G_{S,n}(\psi)(f)\neq 0$.
\end{proof}

\section{Curve complex}\label{S:curve-complex}
Let $S$ be a connected oriented surface (possibly with boundary and punctures). 
A simple closed curve is called essential if it is not isotopic to a boundary curve,
not isotopic to a curve going around exactly one puncture,
and it is not isotopic to a point.

The curve complex $\C C(S)$ of $S$ was first defined by Harvey \cite{MR624817}. This simplicial
complex is defined as follows: for vertices we take isotopy classes of essential
simple closed curves in $S$. A collection of $k+1$ vertices $\{\alpha_i\}_{i=1}^k$
form a $k$-simplex whenever this collection can be realized by pairwise disjoint closed curves in $S$.
A celebrated result of Masur-Minsky states that $\C C(S)$ is hyperbolic \cite{MR1714338}.
We write $\B d_{\C C(S)}$ for the induced combinatorial path-metric on
$\C C(S)$ which assigns unit length to each edge of $\C C(S)$.

The intersection number $\iota_S(\alpha,\beta)$ between two simple closed curves $\alpha,\beta$
on $S$ is defined to be the minimal number of geometric 
intersections between $\alpha'$ and $\beta'$ where $\alpha'$ is 
isotopic to $\alpha$ and $\beta'$ is isotopic to $\beta$. 
Recall that a surface $S$ of genus $g$ with $k$ boundary components
and $n$ punctures is called non-sporadic if $3g+n+k-4>0$. 
Proof of the following lemma may be found in \cite{schleimer}.

\begin{lemma}\label{L:dist-inequality}
Let $S$ be a non-sporadic surface. Then for all simple closed curves 
$\alpha,\beta$ with $\iota_S(\alpha,\beta)\neq 0$ we have
$$\B d_{\C C(S)}(\alpha,\beta)\leq 2\log{\iota_S(\alpha,\beta)}+2.$$
\end{lemma}

\begin{lemma}\label{L:int-number-inequality}
Let $S$ be a compact oriented surface and $p_1,\ldots,p_n\in S$. 
Let $S'=S\setminus\{p_1,\ldots,p_n\}$ and assume that $S'$ is non-sporadic.
Then for every Riemannian metric on $S$ there exists a constant $C$ 
such that for each two essential simple closed curves $\alpha,\beta$ in $S'$ 
we have
$\iota_{S'}(\alpha,\beta)\leq C l(\alpha)l(\beta)$,
where $l(\alpha)$ is the Riemannian length of $\alpha$.
\end{lemma}

\begin{proof}
An analogous statement is proved in \cite{MR568308}, c.f. \cite[Lemma 4.2]{MR3119826}. 
The difference is that there one works with homotopy
classes of curves on the compact surface $S$ and not on the punctured surface.

We construct a specific metric on $S$ such that we are able to use an argument from \cite{MR568308}.
Then, by comparing metrics, the statement of the lemma holds for any Riemannian metric on $S$.
Let $D_i$ be a small disc centered at $p_i$ and let $S_o = S\setminus(D_1 \cup \ldots \cup D_n)$.
We fix a hyperbolic metric on $S_o$  
such that all boundary loops $\partial D_i$ are totally geodesic and have the same length $\epsilon$.
The induced length is denoted by $l_{S_o}$. 

Let $S_r$ be a $2$-dimensional round sphere of radius $r$
and let $B$ be a ball in $S_r$ of perimeter $\epsilon$. We consider $S_{r,o}=S_r\setminus B$. 
By $p$ we denote the point in $S_r$ which is antipodal to the center of the ball $B$.
Let $x$ and $y$ be two different points in $\partial B$ and let $b\subset\partial B$ be an embedded arc
which connects $x$ to $y$.
If the radius $r$ of $S_r$ is big compared to $\epsilon$, then the arc $b$ has the following property. 
Let $\gamma$ be an arc in $S_{r,o}\setminus{p}$ which connects $x$ to $y$. 
Assume that $b$ and $\gamma$ are homotopic in $S_{r,o}\setminus\{p\}$ relatively to $\{x,y\}$. 
Then $l(\gamma)\geq l(b)$, where $l$ is a Riemannian length with respect to the round metric on $S_r$.

Now we construct a metric on $S$. We start with the surface $S_o$. To each boundary 
component $\partial D_i$ we glue a copy of $S_{r,o}$ along the boundary. 
We obtain a surface homeomorphic to $S$. 
Note that in each copy of $S_{r,o}$ there is one antipodal point $p$. 
These antipodal points naturally correspond to points $\{p_i\}_{i=1}^n$.  
On $S$ we consider the path length $l_S$ induced by the hyperbolic 
length $l_{S_o}$ on $S_o$ and round metrics on copies of $S_{r,o}$.

Let $\alpha$ be an essential simple closed curve in $S'=S\setminus\{p_1,\ldots,p_n\}$. 
Since $S_o$ is a deformation retract of $S'$, $\alpha$ is homotopic 
to a simple closed curve that is contained in the hyperbolic surface $S_o$. 
Let $\gamma_{\alpha}$ be the unique hyperbolic geodesic contained in $S_o$ 
which is homotopic to $\alpha$ in $S'$.  

\textbf{Claim.} The loop $\gamma_{\alpha}$ has the minimal length 
among all simple loops homotopic to $\alpha$ in $S'$. 
\begin{proof}
Let $\gamma$ be a simple loop homotopic to $\alpha$ in $S'$. 
Assume, that $\gamma$ is not contained in $S_o$. 
Then there exists a boundary loop of $S_o$, say $\partial D_i$, 
which intersects $\gamma$ in at least two points. 
Let $x$ and $y$ be two distinct points in $\gamma \cap \partial D_i$ 
and let $a$ be the arc contained in $\gamma$ which connects $x$ to $y$
and is disjoint from the interior of $S_o$.
Since $a$ is an embedded arc, 
it is homotopic in $S_{r,o}\setminus\{p_i\}$ relative to $\{x,y\}$ 
to one of the arcs in $\partial D_i$ whose end points are $x$ and $y$. 
Denote this arc by $b$ (see Figure \ref{fig:surface}).
By construction, $l_S(b)\leq l_S(a)$. 
Hence if we substitute $a$ by $b$, we obtain a new loop $\gamma'$, 
which is homotopic to $\gamma$ in $S'$ and $l_S(\gamma') \leq l_S(\gamma)$.
Repeating this procedure we find a loop $\gamma''$ such that $\gamma''\subset S_o$
and $l_S(\gamma'')\leq l_S(\gamma)$. 
Then $l_S(\gamma_{\alpha})\leq l_S(\gamma'')$ and the claim follows. 
\end{proof}

\begin{figure}[htb]
\centerline{\includegraphics[width=0.7\textwidth]{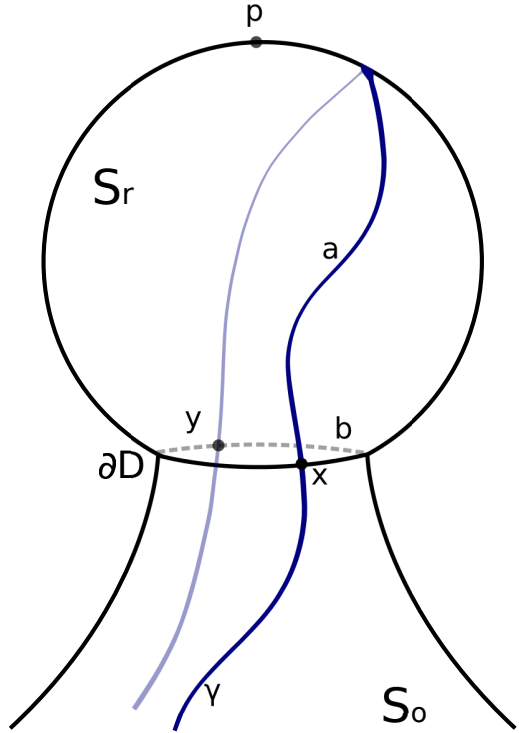}}
\caption{\label{fig:surface} Loop $\gamma$ and arcs $a$ and $b$.}
\end{figure}

Let $\alpha$ and $\beta$ be essential simple loops in $S'$. 
We prove that there is a constant $C$ such that
$$\iota_{S'}(\alpha,\beta) \leq Cl_S(\alpha)l_S(\beta).$$
It follows from the claim, that it is enough to prove that 
$$\iota_{S'}(\alpha,\beta) \leq Cl_S(\gamma_\alpha)l_S(\gamma_\beta).$$ 
Let us repeat the argument from \cite{MR568308}.
We can assume that $\gamma_\alpha \neq \gamma_\beta$, 
otherwise this inequality is trivial. 
Let $r_1$ be a positive number which is less than 
the injectivity radius of the exponential 
map of the surface $S_o$ equipped with the hyperbolic length $l_{S_o}$.
The geodesic $\gamma_\alpha$ may be covered by fewer than 
$\frac{l_S(\gamma_\alpha)}{r_1}+1$ 
geodesic arcs, each of which is contained in a geodesic disc. 
The same holds for $\gamma_\beta$.
Note that if an arc is close to a boundary of $S_o$ such a disk may contain part of the boundary of $S_o$, 
but this does not affect the argument.  
Now a small arc of $\gamma_\alpha$ intersects a small arc of $\gamma_\beta$ in at most one point. 
Thus we have 
$$
\iota_{S'}(\alpha,\beta) \leq \iota_{S'}(\gamma_\alpha,\gamma_\beta)\leq 
\left(\frac{l_{S}(\gamma_\alpha)}{r_1}+1\right)\left(\frac{l_{S}(\gamma_\beta)}{r_1}+1\right)\thinspace .
$$
Since the length $l_{S}$ of every essential simple closed curve 
in $S'$ is greater or equal than $r_1$, we get
$$
\iota_{S'}(\alpha,\beta)\leq \frac{4}{r_1^2} l_S(\gamma_\alpha)l(\gamma_\beta).
$$
Now let ${\bf g}$ be any Riemannian metric on $S$. It is easy to see, 
that the lengths $l$ induced on $S$ by ${\bf g}$ and $l_S$ are comparable. 
Thus, there exists $C$ such that for every loop $\alpha$ we have $l_S(\alpha)<C l(\alpha)$. 
This finishes the proof of the lemma.  
\end{proof}

\section{Mapping class groups}\label{S:MCG} 
Mapping class group $\OP{MCG}(S)$ of an oriented surface $S$ 
is defined to be the group of orientation preserving diffeomorphisms 
of $S$ which fix the boundary pointwise 
modulo diffeomorphisms which are isotopic to the identity. Since an element in 
$\OP{MCG}(S)$ takes homotopy classes of disjoint essential simple closed curves to 
homotopy classes of disjoint essential simple closed curves,
$\OP{MCG}(S)$ acts by isometries on the curve complex $(\C C(S),\B d_{\C C(S)})$.

Let $[f]\in\OP{MCG}(S)$ and $\alpha$ an essential simple closed curve in $S$.
Recall that the translation length of $[f]$ is 
$$\tau_S([f]):=\lim_{p\to\infty}\frac{\B d_{\C C(S)}(f^p(\alpha),\alpha)}{p}.$$
Translation length is independent of the choice of $\alpha$ and 
vanishes on all periodic and reducible elements of $\OP{MCG}(S)$.

\begin{proposition}\label{P:Tau-length-inequality}
Let $S$ be a compact surface and $p_1,\ldots,p_n\in S$. 
Let $S'=S\setminus\{p_1,\ldots,p_n\}$ be a non-sporadic surface.
Let ${\bf{g}}$ be a Riemannian metric on $S$ such that the length of every essential
simple closed curve is greater than one. Then there exists a constant $B$
such that for every $[f]\in\OP{MCG}(S')$ we have
$$\tau_{S'}([f])\leq 2\log{l(f(\alpha))}+B$$
for every essential simple closed curve $\alpha\subset S'$.
\end{proposition}

\begin{proof}
Let $\alpha\subset S'$ be an essential simple closed curve. We have 
$$\tau_{S'}([f])\leq\B d_{\C C(S')}(f(\alpha),\alpha).$$
Note that by the definition of $\C C(S')$ we have $\B d_{\C C(S')}(\alpha,\beta)\leq 1$
if and only if $\iota_{S'}(\alpha,\beta)=0$. We take a constant $C_1:=\max\{C,1\}$, 
where $C$ is a constant from Lemma \ref{L:int-number-inequality}. Now 
Lemma \ref{L:dist-inequality} together with Lemma \ref{L:int-number-inequality}
gives us the following inequality
\begin{align*}
\B d_{\C C(S')}(f(\alpha),\alpha)&\leq 2\log{(C_1 l(f(\alpha))l(\alpha))}+2\\
&=2\log{l(f(\alpha))}+2(\log{(C_1l(\alpha))}+1).
\end{align*}
Combining the last two inequalities we obtain
$$\tau_{S'}([f])\leq 2\log{l(f(\alpha))}+B,$$
where $B=2(\log{(C_1l(\alpha))}+1)$.
\end{proof}

\subsection{Bestvina-Fujiwara quasimorphisms}\label{SS:BF}
Here we describe a construction of quasimorphisms on mapping 
class groups due to Bestvina and Fujiwara \cite{MR1914565}.

Let $S$ be an oriented surface and let $\omega$ 
be a finite oriented path in $\C C(S)$. By $|\omega|$ we denote the length of $\omega$. 
Let $\sigma$ be a finite path. We set
$$
|\sigma|_\omega=\{\rm{the\thinspace\thinspace maximal\thinspace\thinspace number\thinspace\thinspace  
of\thinspace\thinspace  non-overlapping\thinspace\thinspace copies\thinspace\thinspace  
of\thinspace\thinspace} \omega\thinspace\thinspace \rm{in}\thinspace\thinspace \sigma\}.
$$
Let $\alpha,\beta$ be two vertices in $\C C(S)$ and let $W$ be an integer 
such that $0~<~W~<~|\omega|$. Define 
$$c_{\omega,W}(\alpha,\beta)=\B d_{\C C(S)}(\alpha,\beta)-\inf(|\sigma|-W|\sigma|_\omega),$$
where $\sigma$ ranges over all paths from $\alpha$ to $\beta$.

Let $\alpha\in\C C(S)$. We define $\psi_\omega\colon\OP{MCG}(S)\to\B R$ by 
$$\psi_\omega([f])=c_{\omega,W}(\alpha,f(\alpha))-c_{\omega^{-1},W}(\alpha,f(\alpha)).$$
Bestvina and Fujiwara proved that $\psi_\omega$ is a quasimorphism \cite{MR1914565}. 
The induced homogeneous quasimorphism is denoted by $\overline{\psi}_\omega$.
We denote by $Q_{\OP{BF}}(\OP{MCG}(S))$ the space of homogeneous quasimorphisms on $\OP{MCG}(S)$
which is spanned by Bestvina-Fujiwara quasimorphisms. In \cite{MR1914565} it is
proved that $Q_{\OP{BF}}(\OP{MCG}(S))$ is infinite dimensional whenever $S$ is a non-sporadic surface.

Let $i^*\colon Q(\OP{MCG}(S,n))\to Q(\OP{K}(S,n))$ be a homomorphism 
induced by the inclusion map $i\colon K(S,n)\to\OP{MCG}(S,n)$.

\begin{corollary}\label{C:BF-braids}
Let $S$ be a closed oriented surface and $n\in\B N$ such that $S$ with $n$ punctures is non-sporadic. 
Then 
$$\C G_{S,n}(Q_{BF}(\OP{MCG}(S,n)))<Q(\OP{Diff}(S,\OP{area}))$$ 
and
$$\C G^0_{S,n}\circ i^*(Q_{BF}(\OP{MCG}(S,n)))<Q(\OP{Diff}_0(S,\OP{area}))$$ 
are infinite dimensional. 
\end{corollary}

\begin{proof}
The maps $\C G_{S,n}$ and $\C G^0_{S,n}$ are injective by Theorem \ref{T:embed}.
Since the space $Q_{BF}(\OP{MCG}(S,n))$ is infinite dimensional \cite{MR1914565},
it follows that $\C G_{S,n}(Q_{BF}(\OP{MCG}(S,n)))$ is infinite dimensional.

It is left to prove that $i^*(Q_{BF}(\OP{MCG}(S,n)))$ is infinite dimensional.
Since $K(S,n)$ is an infinite normal subgroup of $\OP{MCG}(S,n)$, it is non reducible by
a theorem of Ivanov \cite[Corollary 7.13]{MR1195787}. In order to prove that the space 
$i^*(Q_{BF}(\OP{MCG}(S,n)))$ is infinite dimensional
it is enough to show that $K(S,n)$ is not virtually abelian, 
see \cite[Theorem 12]{MR1914565}. There are three cases.

\textbf{Case 1}. The surface $S$ is a sphere and $n>3$. In this case the mapping class group of $S$
is trivial and thus the group $K(S,n)$ is nothing but $\OP{MCG}(S,n)$ which is not virtually abelian.

\textbf{Case 2}. The surface $S$ is a torus and $n>1$. 
It follows from the Birman sequence for torus \cite{MR0243519}
that the group $K(S,n)$ maps onto $\B B_n(S)$ modulo center, where $\B B_n(S)$ is the torus braid group
on $n$ strings. Let $\B P_n(S)$ be the pure torus braid group on $n$ strings. By removing $n-2$ strings
we get an epimorphism from $\B P_n(S)$ to $\B P_2(S)$ which is isomorphic to $\B Z^2\times \B F_2$. 
It follows that $\B B_n(S)$ modulo center is not virtually abelian, and so is $K(S,n)$.

\textbf{Case 3}. The surface $S$ is hyperbolic. It follows from the Birman 
exact sequence \cite{MR0243519, MR0375281}
that the group $K(S,n)$ is isomorphic to $\B B_n(S)$ which is the braid group of $S$
on $n$ strings. Let $\B P_n(S)$ be the pure braid group of $S$ on $n$ strings. 
By removing $n-1$ strings we get an epimorphism from $\B P_n(S)$ to $\B P_1(S)$ 
which contains $\B F_2$. 
It follows that $\B B_n(S)$ is not virtually abelian, and so is $K(S,n)$.
\end{proof}

\begin{lemma}\label{L:tau-qm-inequality}
Let $S$ be an oriented surface. Then for every quasimorphism
$\psi\in Q_{\OP{BF}}(\OP{MCG}(S))$ there is a positive constant $C_\psi$
such that for every $[f]$ in $\OP{MCG}(S)$ we have
$$|\psi([f])|\leq C_{\psi}\tau_S([f])$$
\end{lemma}
\begin{proof}
It follows from the definition of $Q_{\OP{BF}}(\OP{MCG}(S))$ that for each
$\psi$ in $Q_{\OP{BF}}(\OP{MCG}(S))$ there exist $k\in\B N$,
$a_1,\ldots,a_k\in\B R$ and $\omega_1,\ldots,\omega_k$ finite oriented 
paths in $\C C(S)$ such that 
$$\psi=\sum_{i=1}^k a_i\overline{\psi}_{\omega_i}.$$
Note that we always have $0 \leq c_{\omega,W}(\alpha,\beta) \leq d_{\C C(S)}(\alpha,\beta)$.
Thus 
$$0 \leq \lim_{p\to\infty}\frac{c_{\omega,W}(\alpha,f^p(\alpha))}{p} \leq \tau_S([f]).$$
The limit exists, since the function $p \to c_{\omega,W}(\alpha,f^p(\alpha))$ is expressed as the difference of two
sub-additive functions, namely $p \to d_{\C C(S)}(\alpha,f^p(\alpha))$ and 
$p \to \inf_{\sigma \in P(\alpha,f^p(\alpha))}(|\sigma|-W|\sigma|_\omega)$,
where $P(\alpha,\beta)$ is the set of all paths from $\alpha$ to $\beta$.
By the definition of $\overline{\psi}_{\omega_i}$ we have
$$|\overline{\psi}_{\omega_i}([f])|\leq \tau_S([f]).$$
It follows that
$$|\psi([f])|\leq\left(\sum_{i=1}^k |a_i|\right)\tau_S([f]).$$
\end{proof}

\section{Proofs}\label{S:proofs}
\subsection{Proof of Theorem \ref{T:main}}\label{SS:T1} 
Assume that $S$ is a closed surface. We discuss the case of a surface with boundary at the end. 
Equip $S$ with a metric as in Proposition \ref{P:Tau-length-inequality}.
Denote $\C G=\C G_{S,n}$, $\gamma=\gamma_{S,n}$ 
and assume that $n$ is such that $S$ with $n$ punctures is non-sporadic. 
At the end of the proof we discuss the case when $\C G=\C G^0_{S,n}$.
We pick an essential simple closed curve $\alpha$ in 
the punctured surface $S'=S \setminus\{z_1,\ldots,z_n\}$ (see Section \ref{sec:conf}). 
Let $\psi\in Q_{BF}(\OP{MCG}(S,n))$ and $f\in\OP{Diff}(S,\OP{area})$. 
Then
\begin{align*}
|\C G(\psi)(f)|&\leq\int_{X_n}\lim_{p\to\infty}\frac{|\psi\circ\gamma(f^p,x))|}{p}dx\\
&\leq C_{\psi}\int_{X_n}\lim_{p\to\infty}\frac{\tau_{S'}\circ \gamma(f^p,x)}{p}dx\\
&\leq 2C_{\psi}\int_{X_n}\lim_{p\to\infty}\frac{\log(l(h^{-1}_{f^p(x)}\circ f^p\circ h_x(\alpha)))}{p}dx\thinspace ,
 \end{align*}
where the second inequality is by Lemma \ref{L:tau-qm-inequality}, the third inequality is by 
Proposition \ref{P:Tau-length-inequality}. 

Let $x\in\Omega_n$ (see Section \ref{sec:conf}) and let $U_x$ be an open set such that 
$x\in U_x$ and the closure of $U_x$ is in $\Omega_n$. 
It follows from the Poincar\'e recurrence theorem that for almost all $x$, after passing to a subsequence, 
we can assume that $f^p(x) \in U_x$. 
Due to $C^1$-continuity of the function $h^{-1}_x$ on $\Omega_n$, there exists a constant $K_x$ such that
$$\sup_{p \geq 0}\|h^{-1}_{f^p(x)}\|_1\leq K_x,$$ 
where $\|\cdot\|_1$ is the $C^1$-norm. 
Thus for almost every $x\in X_n$ we get
$$
\lim_{p\to\infty}\frac{\log(l(h^{-1}_{f^p(x)}\circ f^p\circ h_x(\alpha)))}{p}
\leq \lim_{p\to\infty}\frac{\log(K_xl(f^p\circ h_x(\alpha)))}{p}.
$$
This yields
$$
|\C G(\psi)(f)|\leq 2C_{\psi}\int_{X_n}\lim_{p\to\infty}\frac{\log(l(f^p\circ h_x(\alpha)))}{p}dx\thinspace.
$$
We apply Yomdin result \cite[Theorem 1.4]{MR889979} and get that for almost every $x\in X_n$
$$
\lim_{p\to\infty}\frac{\log(l(f^p\circ h_x(\alpha)))}{p}\leq h(f).
$$
Combining last two inequalities we get
$$
|\C G(\psi)(f)|\leq 2C_{\psi}\OP{vol}(X_n)h(f).
$$

Since this inequality applies to any $\psi\in Q_{BF}(\OP{MCG}(S,n))$, 
then by Corollary \ref{C:BF-braids} the space
of quasimorphisms which are Lipschitz with respect to the entropy is infinite dimensional. 

In case when $\C G=\C G^0_{S,n}$ the proof is the same. 
The fact that the space of quasimorphisms on $\OP{Diff}_0(S,\OP{area})$ 
bounding entropy from below is infinite dimensional again follows
from Corollary \ref{C:BF-braids}.

Let us discuss the case of $\OP{Ham}(S)$. It is a simple group which is isomorphic to 
the commutator subgroup of $\OP{Diff}_0(S,\OP{area})$ \cite{MR1445290}.
Since every quasimorphism in $Q_{BF}(\OP{MCG}(S,n))$ vanishes on reducible elements, the space
$i^*\circ Q_{BF}(\OP{MCG}(S,n))$ contains no non-trivial homomorphisms to the reals. In addition, 
for every $\psi\in i^*\circ Q_{BF}(\OP{MCG}(S,n))$ 
the map $\C G^0_{S,n}(\psi)$ is not a homomorphism.
Note that the the kernel of the restriction homomorphism 
$Q(\OP{Diff}_0(S,\OP{area}))\to Q(\OP{Ham}(S))$
is the space $\OP{Hom}(\OP{Diff}_0(S,\OP{area}), \B R)$. Therefore the space
$$\frac{Q(\OP{Diff}_0(S,\OP{area}))}{\OP{Hom}(\OP{Diff}_0(S,\OP{area}), \B R)}$$ 
is isomorphic to a subspace of $Q(\OP{Ham}(S))$. It follows that the map 
$$\C G^0_{S,n}\colon i^*\circ Q_{BF}(\OP{MCG}(S,n))\to Q(\OP{Ham}(S))$$
is injective. Now the proof is identical to the case of $\OP{Diff}_0(S,\OP{area})$.

At last let us comment on the case when $S$ has a boundary. 
In this case we embed $S$ into a closed surface $\overline{S}$.
Then each one of the groups $\OP{Diff}(S,\OP{area})$, $\OP{Diff}_0(S,\OP{area})$ and $\OP{Ham}(S)$
embed in the usual way into the groups $\OP{Diff}(\overline{S},\OP{area})$, 
$\OP{Diff}_0(\overline{S},\OP{area})$ and $\OP{Ham}(\overline{S})$ respectively.
It follows from the proof of the embedding theorem that $\C G_{\overline{S},n}(\psi)$ is
non-trivial on $\OP{Diff}(S,\OP{area})$ provided $\psi$ is non-trivial. Similarly,
$\C G^0_{\overline{S},n}(i^*\circ\psi)$ is
non-trivial on $\OP{Diff}_0(S,\OP{area})$ and on $\OP{Ham}(S)$ provided $i^*\circ\psi$ is non-trivial. \qed

\subsection{Proof of Theorem \ref{T:unboundedness}}
We start with the following

\begin{lemma}\label{L:generation}
Let $G=\OP{Diff}(S,\OP{area})$, $G=\OP{Diff}_0(S,\OP{area})$ or $G=\OP{Ham}(S)$.
Then $G$ is generated by the set $\OP{Ent}(S)$ of entropy-zero diffeomorphisms of~$G$.
\end{lemma}

\begin{proof} \textbf{Case 1:} $G=\OP{Ham}(S)$.
Denote by $\C D$ the set of embedded discs in $S$ of area less than or equal to half of $\OP{area}(S)$. 
Then by fragmentation lemma \cite{MR1445290} for every $f\in G$ there exists a finite collection
of discs $\{D_i\}_{i=1}^k$ in $\C D$ and diffeomorphisms $\{h_i\}_{i=1}^k$ such 
that each $h_i$ is supported in $D_i$ and $f=h_1\circ\ldots\circ h_k$. 
Each $h_i\in\OP{Diff}(D_i,\OP{area})$ is generated
by autonomous diffeomorphisms \cite{MR3044593}. 
An autonomous diffeomorphism is a flow of a vector field, and as such has zero entropy, see \cite{MR0464325}, and the proof follows.

\textbf{Case 2:} $G=\OP{Diff}_0(S,\OP{area})$. 
There is a surjective homomorphism $Flux$ from $G$ to $\frac{\OP{H}^1(S,\B R)}{\Gamma}$, where
$\Gamma$ is the flux group of $Flux$. The kernel of $Flux$ is $\OP{Ham}(S)$ \cite{MR1445290}. 
Take $f\in G$. Then there exists an autonomous (and therefore of zero entropy) 
diffeomorphism $h$ of $S$ such that $Flux(f\circ h)=0$. Hence $f\circ h\in \OP{Ham}(S)$ 
which is generated by entropy-zero diffeomorphisms and so is $G$.

\textbf{Case 3:} $G=\OP{Diff}(S,\OP{area})$. Mapping class group of $S$
is isomorphic to 
$\frac{\OP{Diff}(S,\OP{area})}{\OP{Diff}_0(S,\OP{area})}.$ 
It is generated by Dehn twists. Recall that every Dehn twist has a representative of zero entropy and the 
group $\OP{Diff}_0(S,\OP{area})$ is generated by entropy-zero diffeomorphisms. Hence 
the group $G$ is generated by entropy-zero diffeomorphisms.
\end{proof}



\begin{lemma}\label{L:unb-qm}
Let $G$ be a group, $\C S$ its generating set and $\B d_{\C S}$ 
the induced word metric on $G$.
Let $\psi\colon G\to\B R$ a non-trivial homogeneous
quasimorphism which vanishes on $\C S$. Then 
$$\OP{diam}(G,\B d_{\C S})=\infty.$$ 

\end{lemma}
\begin{proof}
Let $g\in G$. Let $s_1,\ldots,s_k\in\C S$
such that $g=s_1\circ\ldots\circ s_k$ and $\|g\|_{\C S}=k$, where $\|\cdot\|_{\C S}$
is the induced word norm on $G$. Then since $\psi$ vanishes on $\C S$ we have
$$|\psi(g)|=\left|\psi(g)-\sum_{i=1}^k\psi(s_i)\right|\leq D_{\psi}\|g\|_{\C S}.$$
Take $h\in G$ such that $\psi(h)\neq 0$. Then for every $n\in\B N$
$$\|h^n\|_{\C S}\geq n\left(\frac{|\psi(h)|}{D_{\psi}}\right).$$
\end{proof}

It follows from
Theorem \ref{T:main} that there are infinitely many linearly independent homogeneous
quasimorphisms on $\OP{Diff}(S,\OP{area})$, on $\OP{Diff}_0(S,\OP{area})$ and on $\OP{Ham}(S)$ 
which vanish on the set of entropy-zero diffeomorphisms.
By Lemma \ref{L:unb-qm} we have
$$\OP{diam}(\OP{Diff}(S,\OP{area}), \B d_{\OP{Ent}})=\infty,
\quad\OP{diam}(\OP{Diff}_0(S,\OP{area}), \B d_{\OP{Ent}})=\infty\thinspace,$$
$$\OP{diam}(\OP{Ham}(S), \B d_{\OP{Ent}})=\infty\thinspace.$$

Now we prove the second statement of the theorem. Let $S=D^2$ 
be the unit disc in the plane centered at zero and $m\in\B N$.
Note that in this case 
$$\OP{Diff}(D^2,\OP{area})=\OP{Diff}_0(D^2,\OP{area})=\OP{Ham}(D^2).$$

Let $r<\frac{1}{m}$. Denote by $D_r$ the disc in the 
plane of radius $r$ centered at zero. 
Denote $\C G =\C G_{D^2,n}$ and $\C G_r = \C G_{D_r,n}$. 
The inclusion $D_r\subset D^2$ induces an isomorphism $K(D_r,n)\simeq K(D^2,n)$.
Note that $K(D^2,n)=\OP{MCG}(D^2,n)$ which is isomorphic via Birman isomorphism 
to the Artin braid group $\B B_n$. From now on for $x\in X_n$ 
and $f\in\OP{Diff}(D^2,\OP{area}))$ we regard a mapping class $\gamma(f,x)$ as an element of $\B B_n$.
We have:
$$\C G \colon Q(\B B_n) \to Q(\OP{Diff}(D^2,\OP{area}))$$
$$\C G_r \colon Q(\B B_n) \to Q(\OP{Diff}(D_r,\OP{area})).$$
We extend every diffeomorphism in $\OP{Diff}(D_r,\OP{area})$ by identity on $D^2$
and get an injective homomorphism 
$$i_r\colon\OP{Diff}(D_r,\OP{area})\to\OP{Diff}(D^2,\OP{area}).$$

\begin{lemma}\label{L:Ext}
Let $n\geq 4$. Then $\C G_r=i_r^*\circ\C G$
on the linear subspace $Q_{\OP{BF}}(\B B_n)$ of $Q(\B B_n)$.  
\end{lemma}

\begin{proof}
Denote by $X_{n,r}$ the space of all ordered $n$-tuples of
distinct points in $D_r$. 
Let $\psi\in Q_{BF}(\B B_n)$ and $f\in\OP{Diff}(D_r,\OP{area})$.
We have
\begin{align*} 
\C G(\psi)(i_r(f))&=\lim_{p\to\infty}\left(\int\limits_{X_{n,r}}
\frac{\psi(\gamma(f^p;x))}{p}\thinspace dx+
\int\limits_{X_n\setminus X_{n,r}}\frac{\psi(\gamma(f^p;x))}{p}\thinspace dx\right)\\
&=\C G_r(\psi)(f)+\int\limits_{X_n\setminus X_{n,r}}
\lim_{p\to\infty}\frac{\psi(\gamma(f^p;x))}{p}\thinspace dx\thinspace.
\end{align*}
Let $inc\colon\B B_{n-1}\to\B B_n$ 
be the standard inclusion of $\B B_{n-1}$ into $\B B_n$.
Recall that by definition $i_r(f)$ is the identity on $D^2\setminus D_r$.
It follows that for each $x\in X_n\setminus X_{n,r}$ the braid
$$\gamma(f^p;x)=\alpha_{1,p,x}\circ\gamma'_{f^p,x}\circ\alpha_{2,p,x},$$
where the braid $\gamma'_{f^p,x}\in inc(\B B_{n-1})$ and the word
length of the braids
$\alpha_{1,p,x}$ and $\alpha_{2,p,x}$ is bounded for all $p$ and $x$. 
Hence for each $x\in X_n\setminus X_{n,r}$ we have
$$
\lim_{p\to\infty}\frac{\psi(\gamma(f^p;x))}{p}=
\lim_{p\to\infty}\frac{\psi(\gamma'_{f^p,x})}{p}=0,
$$
where the last equality follows from the fact that $inc(\B B_{n-1})$ is
reducible in $\B B_n$ and every quasimorphism in $Q_{\OP{BF}}(\B B_n)$
vanishes on reducible elements. This finishes the proof of the lemma.
\end{proof}

Let us continue the proof.
It follows from Theorem \ref{T:main} that the subspace
$$\C G_r(Q_{\OP{BF}}(\B B_n))<Q(\OP{Diff}(D_r,\OP{area}))$$ 
is infinite dimensional for $n\geq 4$ and that every quasimorphism in this space 
vanishes on the set of entropy-zero diffeomorphisms. It follows from \cite[Lemma 3.10]{MR3044593}
that there exist 
$\{\Psi_{i,n,r}\}_{i=1}^m$ in $\C G_r(Q_{\OP{BF}}(\B B_n))$
and $\{f_{i,n,r}\}_{i=1}^m$ in $\OP{Diff}(D_r,\OP{area})$ such that 
$$\Psi_{i,n,r}(f_{j,n,r})=\delta_{ij},$$
where $\delta_{ij}$ is the Kronecker delta.

Set $f_i:=i_r(f_{i,n,r})$. It follows from Lemma \ref{L:Ext} that 
$\Psi_{i,n}(f_j)=\delta_{ij}$, where $\Psi_{i,n}\in\C G(Q_{\OP{BF}}(\B B_n))$
and it is defined using the same quasimorphism from $Q_{\OP{BF}}(\B B_n)$ as
$\Psi_{i,n,r}$. Each $f_j$ is supported in $D_r$.
Since $r<\frac{1}{m}$ there exists a family of diffeomorphisms 
$\{h_i\}_{i=1}^m\in\OP{Diff}(D^2,\OP{area})$
such that $h_i\circ f_i\circ h_i^{-1}$ and $h_j\circ f_j\circ h_j^{-1}$ 
have disjoint supports for $i\neq j$.
Denote by $\hat{f}_i:=h_i\circ f_i\circ h_i^{-1}$ and let
$$J\colon\B Z^m\to\OP{Diff}(D^2,\OP{area}),$$
where 
$$J(k_1,\ldots,k_m)=\hat{f}_1^{k_1}\ldots\hat{f}_m^{k_m}.$$
It is clear that this map is a monomorphism. We prove it is bi-Lipschitz.
Since all $\hat{f}_i$ commute with each other and $\Psi_{i,n}(\hat{f}_j)=\delta_{ij}$,
we obtain
$$
\|\hat{f}_1^{k_1}\circ\ldots\circ\hat{f}_m^{k_m}\|_{\OP{Ent}}\geq\frac{|\Psi_{i,n}(\hat{f}_1^{k_1}
\circ\ldots\circ\hat{f}_m^{k_m})|}{D_{\Psi_{i,n}}}=\frac{|k_i|}{D_{\Psi_{i,n}}}\thinspace,
$$
where $D_{\Psi_{i,n}}$ is the defect of the quasimorphism
$\Psi_{i,n}$. We denote by $\mathfrak{D}_m:=\max\limits_i D_{\Psi_{i,n}}$
and obtain the following inequality
$$
\|\hat{f}_1^{k_1}\circ\ldots\circ\hat{f}_m^{k_m}\|_{\OP{Ent}}
\geq(m\cdot\mathfrak{D}_m)^{-1}\sum_{i=1}^m |k_i|\thinspace .
$$
Denote by $\mathfrak{M}_J:=\max\limits_i\|\hat{f}_i\|_{\OP{Ent}}$. 
Now we have the following inequality
$$
\|\hat{f}_1^{k_1}\circ\ldots\circ \hat{f}_m^{k_m}\|_{\OP{Ent}}
\leq\sum_{i=1}^m |k_i|\cdot\|\hat{f}_i\|_{\OP{Ent}}
\leq\mathfrak{M}_J\cdot\sum_{i=1}^m |k_i|\thinspace .
$$
Last two inequalities conclude the proof of the theorem. \qed
   
\section{Final remarks}\label{S:Rem}
\begin{enumerate}
\item Let $S$ be a compact oriented surface of positive genus and $n$ such that $S$ with
$n$ punctures is non-sporadic. Since each quasimorphism in $Q_{BF}(\OP{MCG}(S,n))$ vanishes
on reducible elements, the induced quasimorphism on $\OP{Ham}(S)$
vanishes on diffeomorphisms supported in a disc. 
Hence every quasimorphism that lies in the image of 
$\C G^0_{S,n}$ is $C^0$-continuous, see \cite[Theorem 1.7]{MR2884036}.\\
\item Let $S$ be a compact oriented surface. One can easily show 
that there is an infinite family of egg-beater 
Hamiltonian diffeomorphisms $\{f_i\}_{i=1}^\infty$ of $S$
(for definition see \cite{MR3437837}), and a family of linearly independent quasimorphisms
$\Psi_i\in Q(\OP{Diff}(S,\OP{area}))$ which are Lipschitz with respect to the topological
entropy such that $\Psi_i(f_i)\neq 0$. This implies that each $f_i$ has a positive topological entropy.\\
\item 
Since every autonomous diffeomorphism of a surface has zero entropy, 
entropy norm is bounded from above by the autonomous norm. It would be interesting
to know whether these norms are equivalent. Note that the existence of a homogeneous 
quasimorphism on $\OP{Ham}(S)$ which does not vanish on the set of entropy-zero diffeomorphisms,
but vanishes on every autonomous diffeomorphism, would imply that
these norms are not equivalent.
\end{enumerate}

\bibliography{bibliography}
\bibliographystyle{plain}

\end{document}